\documentclass[11pt]{amsart}
\usepackage{amsfonts}
\usepackage{amssymb}
\usepackage{graphicx}
\usepackage{lipsum} 
\usepackage[title]{appendix}
\usepackage{amsmath}
\usepackage{mathrsfs}
\usepackage{hyperref}
\usepackage{verbatim} 

\usepackage{color}\definecolor{Red}{cmyk}{0,1,1,0}

\numberwithin{equation}{section}

\usepackage{tikz, tkz-euclide}
\usetikzlibrary{calc}
\usetikzlibrary{decorations.markings}

\usepackage{bbm}
\definecolor{darkcyan}{rgb}{0.0, 0.55, 0.55}
\newtheorem{teo}{Theorem}

\newtheorem{lema}{Lemma}
\newtheorem{proposition}{Proposition}
\newtheorem{afir}{Affirmation}
\newtheorem{remark}{Remark}
\newtheorem{example}{Example}
\newtheorem{defi}{Definition}

\newcommand{\bbn}{\mathbb{N}}

\newcommand{\al}{\alpha}
\newcommand{\g}{\gamma}
\newcommand{\vt}{\vartheta}

\def\Var{\mathop{\textrm{\rm Var}}\nolimits} 
\def\d{\mathop{\textrm{\rm d}}\nolimits} 
\def\exp{\mathop{\textrm{\rm exp}}\nolimits}

\begin{document}

\voffset=-1.5truecm\hsize=16.5truecm    \vsize=24.truecm
\baselineskip=14pt plus0.1pt minus0.1pt \parindent=12pt
\lineskip=4pt\lineskiplimit=0.1pt      \parskip=0.1pt plus1pt

\def\ds{\displaystyle}\def\st{\scriptstyle}\def\sst{\scriptscriptstyle}

\global\newcount\numsec\global\newcount\numfor
\gdef\profonditastruttura{\dp\strutbox}
\def\senondefinito#1{\expandafter\ifx\csname#1\endcsname\relax}
\def\SIA #1,#2,#3 {\senondefinito{#1#2}
\expandafter\xdef\csname #1#2\endcsname{#3} \else
\write16{???? il simbolo #2 e' gia' stato definito !!!!} \fi}
\def\etichetta(#1){(\veroparagrafo.\veraformula)
\SIA e,#1,(\veroparagrafo.\veraformula)
 \global\advance\numfor by 1
 \write16{ EQ \equ(#1) ha simbolo #1 }}
\def\etichettaa(#1){(A\veroparagrafo.\veraformula)
 \SIA e,#1,(A\veroparagrafo.\veraformula)
 \global\advance\numfor by 1\write16{ EQ \equ(#1) ha simbolo #1 }}
\def\BOZZA{\def\alato(##1){
 {\vtop to \profonditastruttura{\baselineskip
 \profonditastruttura\vss
 \rlap{\kern-\hsize\kern-1.2truecm{$\scriptstyle##1$}}}}}}
\def\alato(#1){}
\def\veroparagrafo{\number\numsec}\def\veraformula{\number\numfor}
\def\Eq(#1){\eqno{\etichetta(#1)\alato(#1)}}
\def\eq(#1){\etichetta(#1)\alato(#1)}
\def\Eqa(#1){\eqno{\etichettaa(#1)\alato(#1)}}
\def\eqa(#1){\etichettaa(#1)\alato(#1)}
\def\equ(#1){\senondefinito{e#1}$\clubsuit$#1\else\csname e#1\endcsname\fi}
\let\EQ=\Eq


\def\\{\noindent}
\def\v{\vskip.1cm}
\def\vv{\vskip.2cm}       
\newcommand\bfblue[1]{\textcolor{blue}{\textbf{#1}}}
\newcommand\blue[1]{\textcolor{blue}{}}
\newcommand\bfred[1]{\textcolor{red}{\textbf{#1}}}
\thispagestyle{empty}

\begin{center}
	{\LARGE\bf Existence of the zero temperature limit of equilibrium states on topologically transitive countable Markov shifts}
	\vskip.5cm
	Elmer R. Beltr\'{a}n$^{1}$, Jorge Littin$^{1}$, Cesar Maldonado$^{2}$ and Victor Vargas$^{3}$
	\vskip.3cm
	\begin{footnotesize}
		$^{1}$Universidad Cat\'olica del Norte, Departamento de Matem\'aticas, Avenida Angamos 0610, Antofagasta - Chile.\\
		$^{2}$ IPICYT / Divisi\'{o}n de Control y Sistemas Din\'{a}micos, Camino a la Presa San Jos\'{e} 2055, Lomas 4a. secci\'{o}n, San Luis Potos\'{i}, S.L.P. - Mexico.\\
		$^{3}$Centre for Mathematics of the University of Porto, Rua do Campo Alegre 687, Porto - Portugal.
	\end{footnotesize}
	\vskip.2cm
	\begin{scriptsize}
		emails: rusbert.unt@gmail.com; jlittin@ucn.cl; cesar.maldonado@ipicyt.edu.mx, vavargascu@gmail.com 
	\end{scriptsize}
\end{center}
\vskip1.0cm
\begin{quote}
{\small

\textbf{Abstract.} \begin{footnotesize} 
		Consider a topologically transitive countable Markov shift $\Sigma$ and a summable Markov potential $\phi$ with finite Gurevich pressure and $\mathrm{Var}_1(\phi) < \infty$. We prove existence of the limit $\lim_{t \to \infty} \mu_t$ in the weak$^\star$ topology, where $\mu_t$ is the unique equilibrium state associated to the potential $t\phi$. Besides that, we present examples where the limit at zero temperature exists for potentials satisfying more general conditions.
	
 \end{footnotesize}}
\end{quote}


\section{Introduction}
Consider $X$ a metric space and let $(X,T)$ be a dynamical system, the thermodynamic formalism studies existence, uniqueness and properties of $T$-invariant probability measures that maximize the value $h(\mu)+\int\phi\d\mu$, where $h(\mu)$ is the metric entropy, those measures are widely known in the mathematical literature as equilibrium states. In this paper we consider $(X, T)$ as a countable Markov shift and $\phi\colon X \to\mathbb{R}$ as a continuous potential. Several properties about these observables were studied using the so called Ruelle operator and variational principles. In particular, in both, the context of finite Markov shifts and countable Markov shifts it were presented the notions of pressure, recurrence and transience in order to guarantee existence and uniqueness, see for instance \cite{BLL13}, \cite{Gu}, \cite{MU} and \cite{Sa1}. 

For any $t\geq 1$, we denote by $\mu_t$ the equilibrium state associated to the potential $t\phi$. An interesting problem in ergodic optimization is to study the weak$^\star$ accumulation points at infinity of the family $(\mu_t)_{t\geq 1}$. The foregoing, because those accumulation points, also called in the mathematical literature as ground states, usually result in maximizing measures for the potential $\phi$, i.e., those measures are the ones giving greater mass to the potential $\phi$ on the set of all the $T$-invariant probability measures. Besides that, the entropy of the  ground states usually satisfy interesting properties among the entropies of all the maximizing measures for the potential $\phi$. Actually, the fact that these accumulation points usually become maximizing measures for the system, shows an interesting connection between thermodynamic formalism and ergodic optimization. 

From the point of view of statistical mechanics, the equilibrium state $\mu_t$ describes the equilibrium of the system whose interactions are given by the potential $\phi$ at temperature $1/t$.  Thus, the existence of the accumulation points of the sequence $(\mu_t)_{t\geq 1}$ is associated with the freezing of the system. Because of that, the accumulation points when $t\to\infty$ are also known as zero temperature limits. A first study about limits at zero temperature in the setting of countable Markov shifts was developed by Z. Coelho in \cite{Coe90}. In fact, in that work, it were studied properties of the pressure associated to the potential $t\phi$ and a version of the central limit theorem in the context of aperiodic finite Markov shifts, also known as topologically mixing shifts of finite type.

When $(X, T)$ is a finite Markov shift, the existence of ground states follows as an immediate consequence of the compactness of the set of Borel probability measures on $X$. On the other hand, the uniqueness of the accumulation point at zero temperature was studied in the setting of locally constant potentials in \cite{Br}, \cite{CGU} and  \cite{Le} assuming transitivity on the dynamics. It is important to mention that J-R. Chazottes and M. Hochman in \cite{CH} exhibited an example where the existence of the limit at zero temperature of the equilibrium state fails when the potential is not locally constant. Another interesting example of existence of more than one accumulation point at zero temperature in the setting of the so called XY models was presented in \cite{vER}. When the alphabet is countable infinite, seminal works about existence of accumulation points at zero temperature were done considering the well known finitely irreducible condition, which is a strong assumption on the combinatorics of the shift $X$ that allowed to generalize several of the main results of the thermodynamic formalism in the countable Markov shifts context, see for instance \cite{JMU1, Ke, MU}. Actually, in \cite{JMU1} it was proved the existence of ground states in the setting of summable potentials. Moreover, the uniqueness was obtained in \cite{Ke} assuming the so called BIP condition on $X$, in the setting of locally constant potentials. Later, in \cite{FV} it was shown the existence of accumulation points at zero temperature under the hypothesis of transitivity. Also, in \cite{SV} it was proved the existence of weak$^\star$ accumulation points at zero temperature for a wider class of positive recurrent potentials defined on topologically transitive countable Markov shifts satisfying suitable conditions. On the other hand, in \cite{LV21} were presented conditions to guarantee existence of ground states in a wider class of linear dynamical systems defined on Banach spaces of infinite dimension.

In this paper we prove uniqueness of the accumulation point at infinity of the family of equilibrium states $(\mu_t)_{t\geq1}$ assuming that the potential $\phi$ has bounded variations and finite Gurevich pressure. The above, serve as a generalization of the results presented in \cite{Ke} to the case of topologically transitive countable Markov shifts. In order to do that, we show that the equilibrium states can be expressed as stationary Markov measures, see \cite{Gu} for details. In addition, we use an approximation of the topologically transitive countable Markov shift by finite Markov subshifts, in a similar way that appears in \cite{Ke}, with the aim of obtaining an approximation of the unique accumulation point at zero temperature of the family $(\mu_t)_{t\geq1}$ by the ones obtained in the setting of finite Markov subshifts. It is important to mention that we assume uniqueness of the accumulation point at zero temperature in the compact context, which was actually proved in \cite{Br}, \cite{CGU} and  \cite{Le}. Additionally, we present some examples where the uniqueness of the accumulation point is guaranteed assuming weaker conditions on the potential $\phi$.

The paper is organized as follows: In Section \ref{Prelim}, we introduce some definitions on thermodynamic formalism in the setting of countable Markov shifts and recall some previously known results. In Section \ref{ZeroTemp}, we study the accumulation points of the family of equilibrium states $(\mu_t)_{t\geq1}$ at infinity and we prove the existence of the zero temperature limit of equilibrium states on topologically transitive countable Markov shifts. Finally, in Section \ref{Examples}, we present two examples of the zero temperature limit of equilibrium state on the renewal shift. 

\section{Preliminaries}\label{Prelim}
Let $S$ be a countable set of {\it{states}} (when $|S|=\infty$ let us consider $S=\mathbb{N}$). Assume that $A=\left(A(i,j)\right)_{S\times S}$ is a square matrix of zeroes and ones with no columns or rows whose entries are all zeroes. Fix the set $\mathbb{N}_0 = \bbn\cup \{0\}$, the {\it{countable Markov shift}} is the set of all the sequences allowed by the matrix $A$, i.e.,
$$\Sigma:=\left\{x=(x_0x_1x_2\ldots)\in S^{\mathbb{N}_0}\colon A(x_i,x_{i+1})=1,~\forall i\geq0\right\},$$
equipped with the topology generated by the collection of {\it{cylinders}}
$$[x_0x_1\ldots x_{n-1}]:=\{y\in \Sigma\colon y_i=x_i,0\leq i\leq n-1\},$$
where $x_i\in S$, for every $0\leq i\leq n-1$ and the shift map (to be defined below) acting on it. The sigma-algebra considered on $\Sigma$ is the smallest one containing all the cylinders, i.e., the Borel $\sigma$-algebra. In the case that the set of states $S$ is finite, the shift $\Sigma$ is known as a finite Markov shift. A {\it path} of length $n$, denoted by $\underline{\gamma}=x_0x_1\ldots x_{n}$, is an element of $S^{n+1}$ satisfying $[x_0x_1\ldots x_{n}]\neq\emptyset$ and we say that the path $\underline{\gamma}$ passes from $x_0$ till $x_n$ through the states $x_1$, ..., $x_{n-1}$. The set of paths of length $n$ is denoted by $\mathcal{P}_n(\Sigma)$ and we denote the set of paths on $\Sigma$ by $\mathcal{P}(\Sigma):=\cup_{n\geq 1}\mathcal{P}_n(\Sigma)$. As usual, the function $\sigma:\Sigma\to\Sigma$ defined by $(\sigma x)_i=x_{i+1}$, for every $i\in\mathbb{N}_0$, is called the {\it{shift map}}.

The countable Markov shift $\Sigma$ is {\it{topologically transitive}} if for every $a,b\in S$ there is a path connecting $a$ and $b$, and it is {\it{topologically mixing}} if there exists $N\in\bbn$ such that there is a path of length $n$ connecting $a$ and $b$, for all $n\geq N$. Also, we say that $\Sigma$ satisfies the {\it{big images and preimages property (BIP)}} if there are $b_1,b_2,\ldots,b_N\in S$ such that, for all $a\in S$, there exists $1\leq i,j\leq N$ such that $A(a,b_i)=A(b_j,a)=1$.

Throughout the paper we call {\it potential} to a continuous function $\phi\colon\Sigma\to\mathbb{R}$ determining the interactions on the lattice (i.e. the one used to define the Ruelle operator). For each $n\geq 1$, we define $S_n\phi(x):=\sum_{i=0}^{n-1}\phi\left(\sigma^i(x)\right)$ as the {\it $n$-th ergodic sum} and the {\it $n$-th variation} of $\phi$ as
$$\Var_n(\phi):=\sup\{|\phi(x)-\phi(y)|\colon x,y\in\Sigma,~x_i=y_i,~0\leq i\leq n-1\}.$$
We say that $\phi$ has {\it bounded variations} if 
$$\overline{\Var}(\phi):=\sum_{n= 1}^{\infty}\Var_n(\phi)<\infty,$$
and has {summable variations} if
$\sum_{n= 2}^{\infty}\Var_n(\phi)<\infty$. Also, $\phi$ is a {\it summable potential} if it satisfies
$$\sum_{i\in\mathbb{N}}\exp\left(\sup\left(\phi|_{[i]}\right)\right)<\infty,$$
where $\sup\left(\phi|_{[i]}\right):=\sup\{\phi(x)\colon x\in[i]\}$, for every $i\in\mathbb{N}$. The so called summability condition becomes important here, because it allows to guarantee a suitable behavior of the Gurevich pressure, also allows to have a uniform control on the tails of the measures belonging to the family $(\mu_{t})_{t\geq1}$ and implies the existence of maximizing measures for the potential $\phi$ (see for instance \cite{BF,JMU1}).

Given a closed $\sigma$-invariant set $\Sigma' \subset \Sigma$,  $\mathcal{M}(\Sigma')$ denotes the set of Borel probability measures on $\Sigma'$, $\mathcal{M}_{\sigma}(\Sigma')$ the set of $\sigma$-invariant Borel probability measures on $\Sigma'$ and $\mathcal{M}_{erg}(\Sigma')$ the set of Borel ergodic probability measures on $\Sigma'$. For any $\mu \in \mathcal{M}(\Sigma')$ we use the following notation 
$$\mu(\phi):=\int_{\Sigma'}\phi\d\mu.$$

For every $\nu\in\mathcal{M}_{\sigma}(\Sigma)$, the {\it metric pressure} is defined by the following quantity 
\begin{equation}\label{metric-pressure} 
	P_\nu := h(\nu)+\nu(\phi),
\end{equation}
where $h(\nu)$ is the metric entropy associated to measure $\nu$ (see \cite{Sa1,Sa3}). The thermodynamic formalism studies the existence and properties of measures $\nu\in\mathcal{M}_{\sigma}(\Sigma)$ that maximize the value of the metric pressure defined in (\ref{metric-pressure}). Note that the sum at the right side of (\ref{metric-pressure}) is not always well defined, the foregoing, because the potential $\phi$ may not be $\nu$-integrable or it could even happen that $h(\nu)=+\infty$ and  $\nu(\phi)=-\infty$. By the above, the usual definition of {\it topological pressure} in the setting of countable Markov shifts is given by 
\begin{equation}\label{prin-var}
	P_{top}(\phi) := \sup\{h(\nu)+\nu(\phi)\colon\nu\in\mathcal{M}_{\sigma}\big(\Sigma\big)\mbox{ s.t. }-\nu(\phi)<\infty\}.
\end{equation}
\noindent
On the other hand, the {\it Gurevich pressure} of $\phi$ is defined by
\begin{equation}\label{Gur}
	P_{G}(\phi):=\limsup\limits_{n\rightarrow\infty}\frac{1}{n}\log Z_n(\phi,a),
\end{equation}
where $Z_n(\phi,a):=\sum_{\sigma^nx=x}\exp\left({S_n\phi(x)}\right)\mathbbm{1}_{[a]}(x)$. It is well known that $P_G(\phi)$ is independent of the choice of $a\in\mathbb{N}$ when the countable Markov shift $\Sigma$ is topologically transitive. Moreover, under the assumptions above $-\infty< P_G(\phi)\leq\infty$. Actually, O. Sarig in \cite{Sa1} showed  that for a topologically mixing countable Markov shift $\Sigma$  and a potential $\phi$ with summable variations and $\sup\phi<\infty$, the Gurevich pressure satisfies the {\it variational principle} and thus
\begin{equation}\label{principle-var} 
	P_G(\phi) = P_{top}(\phi).
\end{equation}
Furthermore, an analogous result in the setting of topologically transitive countable Markov shifts was presented in \cite{BS}.

A measure $\mu\in\mathcal{M}_{\sigma}\big(\Sigma\big)$ is an {\it equilibrium state} associated to the potential $\phi$ if the supremum of \eqref{prin-var} is attained for $\mu$, i.e., when it satisfies
\begin{equation}\label{equil-meas}
	P_\mu = h(\mu)+\mu(\phi) = P_{top}(\phi).
\end{equation}

We say that a potential $\phi$ is {\it recurrent} if $$\sum_{n\geq1}\exp\left({-nP_G(\phi)}\right)Z_n(\phi,a)=\infty.$$
For every $n\geq 1$ and $a\in S$, let
$$
Z_n^{*}(\phi,a):=\sum_{\sigma^nx=x}\exp\left({S_n\phi(x)}\right)\mathbbm{1}_{[\phi_a=n]}(x),
$$
where $\phi_a(x)=\mathbbm{1}_{[a]}(x)\inf\{n\geq 1:\sigma^nx\in[a]\}$ and $\inf \emptyset:=\infty$ (with $0\cdot\infty:=0$). Fix some $a\in S$, we say that a recurrent potential $\phi$ is {\it positive recurrent}, if  
$$\sum_{n\geq1}n\exp\left({-nP_G(\phi)}\right)Z_n^{*}(\phi,a)<\infty.$$
In case that the above series diverges, the potential $\phi$ is called {\it null recurrent}. It is important to point out that when $\Sigma$ is topologically transitive, all modes of recurrence defined above are independent of the choice of $a\in S$, we remit the reader to \cite{Sa3} for details. Also, when $|S|<\infty$ we have that any potential $\phi$ is positive recurrent. The positive recurrent potentials have an important role in the setting in which we are interested in, because they are the ones with an equilibrium state associated to them, as we describe below.

A well known tool in thermodynamic formalism useful to find equilibrium states is the so called {\it Ruelle operator}, which is defined in the setting of countable Markov shifts by the following equation
\begin{equation}\label{op-Ruelle}
	L_{\phi}f(x):=\sum_{\sigma(y)=x}\exp\left({\phi(y)}\right)f(y).
\end{equation} 
When $|S|<\infty$, the Markov shift $\Sigma$ is a compact metric space, the Ruelle operator is well defined on the space of functions $C(\Sigma)$ and we have the famous Ruelle's Perron-Frobenius Theorem, that guarantees the existence of a main eigenvalue with the associated eigenfunctions to the Ruelle operator and eigenmeasures for the dual of the Ruelle operator, see \cite{Bow75,Ru} for complete details.
In general, for countable Markov shifts with $|S|=\mathbb{N}$ the series in (\ref{op-Ruelle}) can be infinite. However, in \cite{Cyr,MU,Sa1,Sh} one can find different types of regularity that can be considered, both on the countable Markov shift $\Sigma$ and on the potential $\phi:\Sigma\to\mathbb{R}$, in order to have the Ruelle operator in (\ref{op-Ruelle}) well defined and obtain an analogous result to the so called Ruelle's Perron-Frobenius Theorem.  

For positive recurrent potentials $\phi$ with $P_G(\phi) < \infty$, O. Sarig showed that there exists a $\phi$-conformal sigma-finite measure $\nu$, that is, a finite Borel measure satisfying  
\begin{equation}\label{dual}	
	\nu\left(L_{\phi}f\right) = \exp(P_G(\phi))\nu(f),\quad \mbox{ for each }f\in L^1(\nu).
\end{equation}
Here the identity in (\ref{dual}) is denoted by $L_{\phi}^{*}\nu=\lambda\nu$ (see Theorem $4.9$ in \cite{Sa3}).

In the context of topologically transitive countable Markov shifts, was proved that for any potential $\phi$ bounded from above, with summable variations and finite Gurevich pressure, there is at most one equilibrium state and, in the case that the existence is guaranteed, the equilibrium state is given by $\d\mu = h\d\nu$, where $h$ is the main eigenfunction of $L_\phi$, i.e., $L_{\phi}h=\exp(P_G(\phi))h$, and $\nu$ is a sigma-finite measure such that $L_{\phi}^{*}\nu=\exp(P_G(\phi))\nu$ (for more details see Theorem $1.1$ and Theorem $1.2$ in \cite{BS})).

\begin{remark} The hypotheses that we assume throughout the paper in order to prove existence of the zero temperature limit are the following ones: we consider $\Sigma$ as a topologically transitive countable Markov shift and $\phi:\Sigma\to\mathbb{R}$ as a summable potential such that ${\overline{\Var}}(\phi)<\infty$ and $P_G(\phi)<\infty$. Under these hypothesis, Theorem $1.1$ from \cite{BS} assures that the Gurevich pressure satisfies the variational principle in (\ref{principle-var}). Moreover, by Theorem $1$ from \cite{FV}, we have that
	$$P_G(t\phi) = P_{top}(t\phi) = h(\mu_t)+t\mu_t(\phi),$$
	for every $t\geq1$, where $\mu_{t}$ is the unique equilibrium state associated to the potential $t\phi$. Besides that, it is also guaranteed the existence of accumulation points at infinity for the family $\left(\mu_{t}\right)_{t\geq1}$.
\end{remark}

Let us define 
\begin{equation}\label{maximizing-value}
	\alpha(\phi):=\sup\left\{\nu(\phi)\colon\nu\in\mathcal{M}_{\sigma}(\Sigma)\right\}.    
\end{equation}
A measure $\mu\in\mathcal{M}_{\sigma}(\Sigma)$ is called $\phi$-{\it maximizing} if $\alpha(\phi)=\mu(\phi)$. We denote by $\mathcal{M}_{max}(\phi)$ the set of $\phi$-maximizing measures.

In the setting of finite Markov shifts, it is widely known that existence of maximizing measures is a direct consequence of the compactness of the subshift. Nevertheless, in the non-compact approach, that is, when the alphabet $S$ is countable infinite, one requires additional conditions on the regularity of the potential.  Indeed, in Theorem $1$ of \cite{BF} the authors proposed conditions on the potential $\phi$ that guarantee the existence of such kind of measures. To be specific, they proved that any coercive potential with bounded variations has a maximizing measure supported on a finite Markov shift $\Sigma_I$, where $I\subset\mathbb{N}$ is a finite set such that $(\Sigma_I,\sigma)$ is a topologically transitive countable Markov shift. Actually, the class of potentials satisfying the so called coercive property considered in \cite{BF}, strictly contains the class of summable potentials. For instance, the potential $\phi(x) := -\log(x_0)$ is coercive but it is not a summable one.

Let $\mathrm{Per}_p(\Sigma)$ be the set of points $x\in\Sigma$ such that $\sigma^p(x)=x$ and consider $\mathrm{Per}(\Sigma):=\bigcup_{p\geq 1}\mathrm{Per}_p(\Sigma)$. For every $x\in\Sigma$, define $\alpha(\phi,x):=\limsup_{n\to\infty}\frac{1}{n}S_n\phi(x)$. So, when $x\in \mathrm{Per}(\Sigma)$, it follows that $\alpha(\phi,x) = \frac{1}{p}S_p\phi(x)$ where $p$ is the period of $x$ (that is, the minimum $p \in \mathbb{N}$ such that $\sigma^p(x)=x$). Denote by $\mathcal{M}_{per}(\Sigma)$ the set of periodic probability measures on $\Sigma$, i.e., the ones supported on periodic orbits. Since
$$
\mathcal{M}_{erg}(\Sigma)=\overline{\mathcal{M}_{per}(\Sigma)}\;.$$
By the ergodic decomposition theorem, we have $\alpha(\phi)=\sup\{\alpha(\phi,x)\colon x\in Per(\Sigma)\}$ (see for instance \cite{BF}). This last identity will be used later in the Example \ref{exem1}.

\section{Zero temperature limits on topologically transitive countable Markov shifts}\label{ZeroTemp}

As we already said in the previous section, for any $t\geq1$, there is a unique equilibrium state $\mu_{t}\in\mathcal{M}_{\sigma}(\Sigma)$ associated to the potential $t\phi$. Moreover, by \cite{FV}, the family of equilibrium states $\left(\mu_{t}\right)_{t\geq1}$ has weak$^\star$ accumulation points at $t\to\infty$. In addition, Theorem \ref{theo-main}, which is the main result of this paper, states that there is at most one of those accumulation points for $\left(\mu_{t}\right)_{t\geq1}$ when $\phi$ is a Markov potential, i.e., $\phi=\phi(x_0x_1)$. The statement of the result is the following:

\begin{teo}\label{theo-main}
	Let $\Sigma$ be a topologically transitive countable Markov shift and let $\phi\colon\Sigma\to\mathbb{R}$ be a summable Markov potential such that ${\Var}_1(\phi)<\infty$ and $P_G(\phi)<\infty$. Then, the limit $\lim_{t\to\infty}\mu_t$ exists in the weak$^\star$ topology, where $\mu_{t}\in\mathcal{M}_{\sigma}(\Sigma)$ is the unique equilibrium state associated to the potential $t\phi$, for  $t\geq1$. Furthermore, the limit measure $\mu_\infty$ is $\phi$-maximizing.
\end{teo}

The main idea to prove the above theorem is to obtain an approximation of the weak$^\star$ accumulation points at $\infty$ of the family $(\mu_t)_{t\geq1}$ by the unique accumulation point at $\infty$ of a family of equilibrium states $(\vt_{t})_{t\geq1}$ defined on a suitable finite Markov shift $\Sigma'$, which guarantees uniqueness in the non-compact context. Here we use a similar technique to the one in \cite{Ke} guaranteeing that in our setting any equilibrium state can be expressed as a stationary Markov measure and using that expression to obtain the desired approximation.

Our first goal here is to characterize the behavior of those accumulation points and the asymptotic behavior of the map $t \mapsto P_G(t\phi)$. In order to do that, next we present a lemma that will be useful to state and prove Proposition \ref{acum-max}, which assures that every weak$^\star$ accumulation point at $\infty$ of the family $\left(\mu_{t}\right)_{t\geq1}$ is a maximizing measure for the potential $\phi$. The following lemma is already a well known result in the matter of countable Markov shifts satisfying the BIP condition, for details see \cite{BM}, and its validity basically depends on the existence of equilibrium states $\mu_t$, for each $t\geq 1$. Since, the existence of each one of the $\mu_t$'s is guaranteed in the topologically transitive context (see \cite{FV}), it is expected to obtain something similar to the result in \cite{BM} here. The statement of the lemma is the following:

\begin{lema}\label{prop-Esga-0}
	Given $t \geq 1$, assume that $\mu_t\in\mathcal{M}_{\sigma}(\Sigma)$ is an equilibrium state associated to the potential $t\phi$. Then,
	\begin{itemize}
		\item[i)] The family $\{h(\mu_t)\}_{t\geq 1}$ is decreasing,
		\item[ii)] $\lim\limits_{t\to\infty}\frac{P_G(t\phi)}{t}=\alpha(\phi)$.
	\end{itemize}
\end{lema}
\begin{proof} A similar procedure to the one in Lemma $9$ from \cite{BM} shows that the family $\{\mu_{t}(\phi)\}_{t\geq1}$ is increasing. 
	Fix $1\leq t_1<t_2$. Since $\mu_1\in\mathcal{M}_{\sigma}(\Sigma)$ is an equilibrium state for the potential $t_1\phi$, we have
	\begin{eqnarray}\label{entrop-1}
		h(\mu_{t_1})=P_G(t_1\phi)-t_1\mu_{t_1}(\phi).
	\end{eqnarray}
	Note that by the variational principle in (\ref{principle-var}), we have $P_G(t_1\phi)> h(\mu_{t_2})+t_1\mu_{t_2}(\phi)$. Hence, replacing in (\ref{entrop-1}), we obtain
	\begin{equation}\label{entrop-2}
		h(\mu_{t_1})> h(\mu_{t_2})+t_1\left(\mu_{t_2}(\phi)-\mu_{t_1}(\phi)\right).
	\end{equation}
	As $\mu_{t_2}(\phi)-\mu_{t_1}(\phi)> 0$, so $h(\mu_{t_1})> h(\mu_{t_2})$ and thus the family $\{h(\mu_t)\}_{t\geq 1}$ is decreasing. 
	
	Now we will prove item $ii)$. Note that $|\mathcal{M}_{max}(\phi)|\neq\emptyset$, see Theorem $1$ from \cite{BF}. Let $\nu\in\mathcal{M}_{max}(\phi)$ and $t\geq 1$, then by the variational principle (\ref{principle-var}), we have the following
	$$h(\nu)+t\nu(\phi)\leq \sup\left\{h(\mu)+t \mu(\phi)\colon\mu\in\mathcal{M}_{\sigma}(\Sigma)\;\mbox{ and }- \mu(\phi)<\infty\right\}=P_{G}(t\phi).$$
	Later
	\begin{equation}\label{max-pres}
		t\alpha(\phi)\leq P_G(t\phi).
	\end{equation}
	On the other hand, the map $t\mapsto P_G(t{\phi})-t\alpha(\phi)$ is decreasing in $[1,\infty)$, see \cite{BM}. Therefore $0\leq P_G(t{\phi})-t\alpha(\phi)\leq P_G({\phi})-\alpha(\phi)<\infty$ for every $t \geq 1$, so 
	$$\lim\limits_{t\to\infty}\frac{P_G(t\phi)}{t}=\lim\limits_{t\to\infty}\frac{1}{t}\left(P_G(t\phi)-t\alpha(\phi)\right)+\alpha(\phi)=\alpha(\phi).$$
\end{proof} 

\begin{proposition}\label{acum-max}
	Every weak$^\star$ accumulation point at $\infty$ of the family of equilibrium states $\left(\mu_{t}\right)_{t\geq1}$ belongs to the set $\mathcal{M}_{max}(\phi)$.
\end{proposition}
\begin{proof} Consider some arbitrary accumulation point $\mu_{\infty}\in\mathcal{M}_{\sigma}(\Sigma)$ of the family $(\mu_{t})_{t\geq1}$. It is enough to verify that
	\begin{equation*}
		\al(\phi)\leq\mu_{\infty}\left(\phi\right).
	\end{equation*}
	Indeed, since that map $\mu\mapsto\mu(\phi)$, from $\mathcal{M}_{\sigma}(\Sigma)$ into $[0, \infty)$, is upper semi-continuous in the weak$^\star$ topology, see Lemma $1$ \cite{JMU1}, we have 
	\begin{equation}\label{equ-max1}
		\limsup_{t\to\infty}\mu_t(\phi)\leq \mu_{\infty}(\phi).
	\end{equation}
	
	\noindent
	On the other hand, by variational principle, for each $t\geq1$ we have
	$$\frac{P_G(t\phi)}{t}=\frac{h(\mu_t)}{t}+\mu_{t}(\phi).$$
	Now, taking the $\limsup$ in the last equality, since $\left(h(\mu_t)\right)_{t\geq 1}$ it is bounded by above (see Lemma \ref{prop-Esga-0}), we obtain that
	$$\alpha(\phi)=\limsup\limits_{t\to\infty}\frac{P_G(t\phi)}{t}\leq\limsup_{t\to\infty}\mu_{t}(\phi)\leq \mu_{\infty}(\phi).$$
\end{proof}

An important condition necessary to prove the convergence in the weak$^\star$ topology of the family of equilibrium states $(\mu_t)_{t\geq1}$ at $\infty$, is to show that the potential $t\phi$ is positively recurrent for each $t\geq1$. That statement is verified in the following proposition. 

\begin{proposition}\label{lema-posrecu}
	For every $t\geq 1$ the potential $t\phi$ is positive recurrent.
\end{proposition}
\begin{proof} Note that the potential $\phi$ is positive recurrent. The Theorem 1  in \cite{FV} guarantees that for every $t>1$ there is a unique equilibrium state $\mu_{t}$ associated to potential $t\phi$, so by Theorem $1.2$ in \cite{BS} the potential $t\phi$ is positive recurrent, for every $t>1$.
\end{proof}

\begin{remark}
	Under the same hypotheses of the previous lemma, but assuming that $\Sigma$ is a finitely primitive countable Markov shift, I. Morris showed in \cite{Mo}  that
	\begin{equation}\label{entrop-inf}
		h(\mu_{\infty})=\lim_{t\to\infty}h(\mu_t)=\sup_{\nu\in\mathcal{M}_{max}(\phi)}h(\nu),
	\end{equation}
	where $\mu_{\infty}\in\mathcal{M}_{\sigma}(\Sigma)$ is some accumulation point of the family of equilibrium states $(\mu_t)_{t\geq1}$. Actually, R. Freire and V. Vargas in \cite{FV} obtained an extension of \eqref{entrop-inf} for the setting of topologically transitive countable Markov shifts.
\end{remark}

For every potential $\phi$, the potential $\widetilde{\phi}=\phi-\alpha(\phi)$ is called {\it normalized potential}. Notice that $\widetilde{\phi}\leq 0$ and $\alpha(\widetilde{\phi})=0$. Furthermore, those potentials have the same equilibrium state, that is, $\mu_{\phi}=\mu_{\widetilde{\phi}}$. So, for ease of computation, from now on, we consider the normalization $\widetilde{\phi}$ of $\phi$ and, in order to not overload the notation, we will denote simply by $\phi$ the normalized one.

\begin{proposition}\label{pres-decre} The following properties are satisfied
	\begin{itemize}
		\item[i)] $P_G(t\phi)\geq0$, for every $t\geq1$;
		\item[ii)] the function $t\mapsto P_G(t\phi)$ is decreasing;
		\item[iii)] $\lim\limits_{t\to\infty}P_{G}(t\phi)=h(\mu_{\infty})$, where $\mu_{\infty}\in\mathcal{M}_{max}(\Sigma)$ is a weak$^\star$ accumulation point at $\infty$ for the family of equilibrium states $\left(\mu_{t}\right)_{t\geq1}$.
	\end{itemize}
\end{proposition}
\begin{proof} The proofs of items $i)$ and $ii)$ are obtained directly from Lemma \ref{prop-Esga-0}. Now we proceed to check item $iii)$.  Let $\{t_k\}_{k\in\mathbb{N}}$ be an increasing sequence of numbers greater than one converging to infinity such that $\mu_{t_k}\to\mu_{\infty}$ as $k\to\infty$ in the weak$^\star$ topology. Note that $\lim_{k\to\infty}\mu_{t_k}(\phi)=0$, see Proposition \ref{acum-max}. So, by Theorem 2 \cite{FV}, we obtain
	\begin{eqnarray*}\label{pres-to-h}
		h(\mu_{\infty})&=&\limsup_{k\to\infty}h(\mu_{t_k})\\
		&=&\limsup_{k\to\infty}\left(P_G(t_{k}\phi)-t_{k}\cdot\mu_{t_k}(\phi)\right)\\
		&\geq&\limsup_{k\to\infty}\left(P_G(t_k\phi)-\mu_{t_k}(\phi)\right)\\
		&=&\lim_{k\to\infty}P_G(t_k\phi)-\liminf_{k\to\infty}\mu_{t_k}(\phi)\\
		&=&\lim_{t\to\infty}P_G(t\phi).
	\end{eqnarray*}
	On the other hand $P_G(t\phi)\geq h(\mu_{\infty})+t\mu_{\infty}(\phi)=h(\mu_{\infty})$. Therefore $\lim\limits_{t\to\infty}P_{G}(t\phi)=h(\mu_{\infty})$.
\end{proof}

\subsection{Existence of the zero temperature limit for Markov potentials}

From now on, we consider locally constant potentials $\phi: X \to \mathbb{R}$, so that without loss of generality we can assume $\phi$ as a Markov potential, i.e., $\phi(x)=\phi(x_0 x_1)$. This is true because any locally constant potential is cohomologous to a Markov potential (we send the reader to \cite{JMU2} for details).

By \cite{Sa1}, we have that the potential $\phi(x) = \phi(x_0 x_1)$ has an associated equilibrium state $\mu\in\mathcal{M}_{\sigma}(\Sigma)$ because it is positive recurrent, see Lemma \ref{lemma1-Kem}. So the Ruelle operator  $L_\phi$ is well defined on the space of bounded continuous functions. In particular, for any function of the form $\psi(x)= \psi(x_0)$ we have
\begin{equation}\label{op-Ruel-h}
	L_\phi (\psi)(x) = L_\phi (\psi)(x_0) = \sum_{\substack{a\in S\\ A(a, x_0)=1}}\exp({\phi(ax_0)})\psi(a). 
\end{equation}
From Theorem $1.2$ in \cite{BS}, it follows that the operator $L_{\phi}$ has a strictly positive eigenfunction $h$, $L_{\phi}h=\lambda h$, where $\lambda=\exp(P_{top}(\phi))$ and by (\ref{op-Ruel-h}) we have $h(x)=h(x_0)$. In this case, we also can define the transpose of the Ruelle operator, $L^\intercal_{\phi},$ calculated in a function $\psi(x)= \psi(x_0)$ as 
$$
L^\intercal_\phi (\psi)(x) = L^\intercal_\phi (\psi)(x_0) = \sum_{\substack{a\in S \\ A(x_0 a)=1}}\exp({\phi(x_0a)})\psi(a).
$$
It is not difficult to check that the operator $L^\intercal_\phi$ has a strictly positive eigenfunction $h^\intercal$, satisfying $L^\intercal_\phi(h^\intercal) = \lambda h^\intercal$ and $h^\intercal(x) = h^\intercal(x_0)$, where $\lambda$ is the main eigenvalue of the operator $L_\phi$, and $\sum_{a \in S}h(a)h^\intercal(a) = 1$. A detailed proof about this claim can be found in \cite{Gu}. 

\begin{remark}\label{remark-G}
	Theorem C in \cite{Gu} states that the equilibrium state $\mu$ is unique and it is an stationary Markov measure given by the formula
	\begin{equation}\label{Markov1}
		\mu([x_0x_1\ldots x_n])=\pi(x_0)p(x_0 x_1)p(x_1 x_2)\ldots p(x_{n-1} x_n),
	\end{equation}
	where $\pi(a)=h(a)h^\intercal(a)>0$ for all $a \in S$, is the  stationary probability measure. Moreover, here $h$, $h^\intercal$ are the main eigenfunctions of the previously indicated operators $L_\phi$ and $L^\intercal_\phi$, respectively. The explicit form for the transition probabilities is given by
	\begin{equation}\label{Markov2}
		p(a,b)= \frac{h\left(b\right)}{ h(a)}\exp\left(\phi(ab)-P_{top}(\phi)\right) .  
	\end{equation}
\end{remark}

Here, it is convenient to define a measure $\widehat{\mu}_t\in\mathcal{M}_{\sigma}(\widehat{\Sigma})$ on the bilateral countable Markov shift $\widehat{\Sigma} := \{x \in S^{\mathbb{Z}}\colon A(x_i, x_{i+1})=1,\; \forall i \in \mathbb{Z}\}$, associated to the potential $t\widehat{\phi}((x_i)_{i \in \mathbb{Z}}) = t\phi(x_0x_1)$, given by $\widehat{\mu}_t([x_mx_{m-1}\ldots x_n]) = \mu_t([x^{\prime}_0x^{\prime}_1\ldots x^{\prime}_{n-m}])$, $m, n \in \mathbb{Z}$ and $m \leq n$, where $x^{\prime}_0 =x_m,\ldots, x^{\prime}_{n-m}=x_n$. The measure $\widehat{\mu}_t$ is invariant under the bilateral shift map and convergence of the family $(\widehat{\mu}_t)_{t\geq1}$ implies the convergence of the family $(\mu_t)_{t\geq1}$. 

To facilitate the computations that appear below, we will use the measures $\widehat{\mu}_t$ defined on the bilateral countable Markov shift $\widehat{\Sigma}$ instead of the measures $\mu_t$ defined on the unilateral one $\Sigma$. Trying to not overload the notation, hereafter we will denote the bilateral countable Markov shift by $\Sigma$, we will use the notation $\mu_t$ for its corresponding equilibrium states and we will denote by $\sigma$ the map given by $(\sigma x)_i = x_{i+1}$ for any $i \in \mathbb{Z}$.

Also, in order to simplify our notation, for each path $\underline{\g}=x_0x_1\ldots x_n$ we use $l(\underline{\g})$ to denote its length and  $\phi(\underline{\g}):=S_n\phi(x)$, $x\in[\underline{\gamma}]$. Since $\phi$ is a Markov potential, it follows that $\phi(\underline{\g})$ is constant on each  $x\in[\underline{\g}]$, so this notation is not ambiguous. Similarly, for any probability measure $\mu\in\mathcal{M}_{\sigma}(\Sigma)$, we write $\mu(\underline{\g}):=\mu\left([\underline{\g}]\right)$.  
Now, for a typical path $\underline{\gamma}=x_0x_1\ldots x_n \in \mathcal{P}(\Sigma)$, from equations \eqref{Markov1} and \eqref{Markov2} of Remark \ref{remark-G} we have
\begin{eqnarray}
	\mu_t(\underline{\gamma})&=&\pi(x_0)\prod_{k=0}^{n-1}\frac{h\left(x_{k+1}\right)}{ h(x_k)}\exp\Big(t\phi(x_kx_{k+1})-P_G(t \phi)\Big)\nonumber\\
	&=&\pi(x_0) \frac{h\left(x_{n}\right)}{ h(x_0)}\exp\left(t\phi(\underline{\gamma})-nP_G(t \phi)\right).\label{med-mark}
\end{eqnarray}
Obviously $\pi(x_0)=\mu_t(x_0)>0$, when  $\underline{\gamma}$ is a loop, that is $x_0=x_n$ we have $h(x_0)=h(x_n)$ and consequently
\begin{equation}\label{Kemptoneq}
	\mu_t(\underline{\gamma}) = \mu_t(x_0) \exp\left(t\phi(\underline{\gamma})-nP_G(t \phi)\right),
\end{equation}
this identity was deduced by Kempton in a more restrictive case (see  \cite{Ke}). In fact, the positive recurrence of the potential $t \phi$ and the topologically transitive condition of the Markov shift $\Sigma$ are necessary and sufficient conditions to get \eqref{Kemptoneq} (see Theorems C and D from \cite{Gu} for more details). 

Notice that by the notation introduced earlier, we have that
$$
\left(t\phi-P_{G}(t\phi)\right)(\underline{\g})=t\phi(\underline{\g})-nP_{G}(t\phi),$$
for every $t\geq 1$. Therefore, for any loop $\underline{\g}=x_0x_1\ldots x_n$ satisfying $x_0=x_n$, the equation \eqref{Kemptoneq} can be re-written into the form 
\begin{equation}\label{coc-loop}
	\frac{\mu_{t}(\underline{\gamma})}{\mu_{t}([x_0])}= \exp\left((t\phi-P_G(t \phi))(\underline{\gamma})\right).
\end{equation}
\noindent
On the other hand, by Theorem $1$ from \cite{BF} there exists a finite set $I\subset\mathbb{N}$, such that any $\phi$-maximizing measure $\mu\in\mathcal{M}_{\sigma}(\Sigma)$ satisfies $supp(\mu) \subset \Sigma_I$. From now on, the finite set $I$ will denote the set given for this theorem. 

\begin{remark}\label{d}
	Since $\phi$ is a coercive potential, by Lemma $2$ in \cite{BF} there exists $d>0$ such that $\sup \phi|_{[i]}<-d$, for every $i\notin I$, where $\sup\phi|_{[i]}:=\sup\{\phi(x)\colon x\in[i]\}$. 
\end{remark}

By Proposition \ref{acum-max}, we have that any weak$^\star$ accumulation point $\mu_\infty$ of the family of equilibrium states $(\mu_t)_{t\geq1}$ is a maximizing measure supported on $\cup_{i \in I}[i]$. As a consequence, the existence of the limit $\lim_{t \to \infty}\mu_t$ in the weak$^\star$ topology is equivalent to showing existence of the limit $\lim_{t\to\infty}\mu_t[a]$ for all $a\in I$ (see for instance \cite{Ke}). Because of that, it is enough to check the convergence of the ratios $\lim_{t\to\infty}\frac{\mu_{t}([b])}{\mu_{t}([a])}$, for all $a,b\in I$, to show the convergence of $(\mu_t)_{t\geq1}$ in the weak$^\star$ topology. In fact, the limit of the ratios can even be infinite.

From now on let us fix $a,b \in  I$, we define 
$$\Sigma(a) := \{x\in\Sigma\colon x_i =a \mbox{ for infinitely many } i \in \mathbb{N}_0\} \;.$$
Clearly $\Sigma(a)$ is a closed $\sigma$-invariant countable Markov subshift. For every $t \geq 1$, the potential $t \phi$ is positive recurrent, see Lemma \ref{lema-posrecu}, so that $\mu_{t}\in\mathcal{M}_{\sigma}(\Sigma)$ is an ergodic measure and moreover $\mu_{t}(\Sigma(a))=1$, for all $t\geq1$. Then we have that $\mu_{t}([b])=\mu_{t}(\Sigma(a)\cap[b])$, for every $b \in I$.

Let $\Gamma(a)$ denote the set of paths $\underline{\gamma}=x_0x_1\ldots x_n$, $n\geq 1$, such that $x_j=a$ iff $j\in\{0,n\}$. Since $\Gamma(a)$ is countable (because is countable union of countable sets), this allows write them as $\Gamma(a)=\{\underline{\gamma_i}\}_{i=1}^{\infty}$, where every $\underline{\gamma}_i\in \Gamma(a)$ for $i\in\mathbb{N}$. Notice that  $\Sigma(a)$ can be splitted as
$$\Sigma(a)=\bigcup_{i=1}^{\infty}\bigcup_{k=1}^{l(\underline{\gamma}_i)}\sigma^k[\underline{\gamma}_i],$$
and so,
\begin{equation}\label{a-inter-b}
	\Sigma(a)\cap[b]=\bigcup_{i=1}^{\infty}\bigcup_{k=1}^{l(\underline{\gamma}_i)}\sigma^k[\underline{\gamma}_i]\cap[b].  
\end{equation}
For any loop $\underline{\gamma}_i \in \Gamma(a)$, let $N(b,\underline{\gamma}_i)$  be the number of occurrences of the symbol $b$ within the loop $\underline{\gamma}_i$, note that  $N(b,\underline{\gamma}_i)=\sum_{k=1}^{l(\underline{\gamma}_i)}\mathbbm{1}_{[b]}\left(\sigma^k[\underline{\gamma}_i]\right)$.

Fix $t\geq 1$, recalling that $\mu_{t}([b])=\mu_{t}\left(\Sigma(a)\cap[b]\right)$ and $\mu_t$ is invariant by the action of the bilateral shift $\sigma$, from equation (\ref{a-inter-b}) we see that 
\begin{eqnarray}
	\mu_{t}([b])&=&\sum_{i=1}^{\infty}\sum_{k=1}^{l(\underline{\gamma}_i)}\mu_{t}\left(\sigma^k[\underline{\gamma}_i]\right)\mathbbm{1}_{[b]}\left(\sigma^k[\underline{\gamma}_i]\right)\\
	&=&\sum_{i=1}^{\infty}\mu_{t}(\underline{\gamma}_i)N(b,\underline{\gamma}_i).
\end{eqnarray}
From \eqref{coc-loop} we know that for any closed loop $\underline{\gamma}_i\in \Gamma(a)$ 
\begin{equation}\label{coc-loop2}
	{\mu_{t}(\underline{\gamma}_i)}={\mu_{t}([a])}\exp\left((t\phi-P_G(t \phi))(\underline{\gamma}_i)\right),
\end{equation}
so
\begin{eqnarray}
	\mu_{t}([b])&=&\sum_{i=1}^{\infty}\mu_{t}([a])\exp\left((t\phi-P_{G}(t\phi))(\underline{\gamma}_i)\right)N(b,\underline{\gamma}_i)
\end{eqnarray}
and hence
\begin{equation}\label{b-a}
	\frac{\mu_{t}([b])}{\mu_{t}([a])}=\sum_{i=1}^{\infty}\exp\left((t\phi-P_{G}(t\phi))(\underline{\gamma}_i)\right)N(b,\underline{\gamma}_i). 
\end{equation}
Actually, the finiteness of $\frac{\mu_{t}([b])}{\mu_{t}([a])}$, for all $t \geq 1$ is guaranteed by the positive recurrence of the potential $t \phi$ (see for instance \cite{Sa1}).

Those closed loops $\underline{\gamma}_i\in\Gamma(a)$ which do not pass through of the symbol $b$ have no relevance at the right-hand of the equation \eqref{b-a}, because $N(b,\underline{\gamma}_i)=0$. In this case, the equation \eqref{b-a} is equivalent to
\begin{equation}\label{b-a2}
	\frac{\mu_{t}([b])}{\mu_{t}([a])}=\sum_{m=1}^{\infty}m \sum_{\substack{\underline{\g}_i \in \Gamma(a)}}\exp\left((t\phi-P_{G}(t\phi))(\underline{\gamma}_i)\right)\mathbbm{1}_{[N(b,\underline {\g}_i)=m]},
\end{equation}
where $\mathbbm{1}_{[N(b,\underline {\g}_i)=m]}=1$ if $N(b,\underline {\g}_i)=m$ and otherwise $\mathbbm{1}_{[N(b,\underline {\g}_i)=m]}=0$.

\begin{defi}
	Let $a,b\in I$. We say that $\underline{\g}=x_0\ldots x_n$, $n \geq 1$, is a main path in $\Sigma$ that starts at $i$ and ends at $j$ where $i,j\in\{a,b\}$ iff $x_0=i$, $x_n=j$ and $x_m\notin\{a,b\}$ for $m\in \{1,2,\ldots,n-1\}$. This means that the symbols $a,b$ does not appear in the middle of the path $\underline{\g}$. We will denote by $\{\underline{\g}\colon i\to j\}$ the set of all the main paths that start at $i$ and end at $j$. 
\end{defi}

Note that $\{\underline{\g}\colon a\to a\}$ and $\Gamma(a)$  do not represent the same set. The foregoing is true because $\Gamma(a)$ contains paths with the symbol $b$ while $\{\underline{\g}\colon a\to a\}$ does not. For $\widetilde{\Sigma}$ a topologically transitive Markov subshift of $\Sigma$, i.e., $\widetilde{\Sigma}\subset \Sigma$ closed and invariant by $\sigma$, we will use the following notation $\{\underline{\g}\colon i\to j\in\mathcal{P}(\widetilde{\Sigma})\}$ to indicate that each main path that starts at $i$ and ends at $j$ is a path of $\widetilde{\Sigma}$ (remember that $\mathcal{P}(\widetilde{\Sigma})$ denotes the set of paths in $\widetilde{\Sigma}$).

For every $i,j \in \{a,b\}$ we define 
\begin{equation}\label{p_}
	p_{ij}^{t}:=\sum_{\underline{\gamma} \in\{\underline{\g}\colon i\to j\}}\exp\left((t\phi-P_{G}(t\phi))(\underline{\gamma})\right).
\end{equation}
Note from (\ref{coc-loop}) that $p_{ii}^{t}<1$, for every $i\in I$. Indeed, the probability, with respect to $\mu_{t}$, that a path from $i$ returns to $i$ eventually is one, so from (\ref{coc-loop}) and (\ref{p_}) we observe that it can be splitted into $p_{ii}^{t}$, the probability that a path from $i$ returns to $i$ without passing through $j$, where $j\neq i$, and $p_{ij}^{t}\left(\sum_{n\geq 1}(p_{jj}^{t})^n\right)p_{ji}^{t}$, the probability that a path from $i$ returns to $i$ passing through $j$ at least once. Therefore
\begin{equation}\label{A}
	p_{ii}^{t}+p_{ij}^{t}p_{ji}^{t}\sum_{n\geq 1}(p_{jj}^{t})^n=1.
\end{equation}
\noindent
Recalling that $t \phi$ is a Markov potential, for all $m\in\mathbb{N}$ we get from \eqref{p_}
\begin{eqnarray}
	\sum_{\substack{\underline{\g} \in\Gamma(a) }}\exp\left((t\phi-P_{G}(t\phi))(\underline{\gamma})\right)\mathbbm{1}_{[N(b,\underline {\g})=m]}=p_{ab}^t(p_{bb}^t)^{m-1}p_{ba}^t.
\end{eqnarray}
Therefore, from \eqref{b-a2} we obtain 
\begin{equation}\label{b-a3}
	\frac{\mu_{t}([b])}{\mu_{t}([a])}=\sum_{m=1}^{\infty}m p^t_{ab}(p_{bb}^t)^{m-1}p^t_{ba}=\frac{p^t_{ab}p^t_{ba}}{(1-p^t_{bb})^2}.
\end{equation}
\noindent
By inverting the roles of $a$ and $b$ we get $\frac{\mu_{t}([a])}{\mu_{t}([b])}=\frac{p^t_{ba}p^t_{ab}}{(1-p^t_{aa})^2}$ and combining both expressions we conclude
\begin{equation}\label{coc-mu}
	\frac{\mu_{t}([b])}{\mu_{t}([a])}=\frac{1-p_{aa}^t}{1-p_{bb}^t}.
\end{equation}
\noindent
Therefore, proving the convergence of the equilibrium states $(\mu_t)_{t\geq 1}$, $\lim_{t\to \infty}\mu_{t}$, reduces to showing the existence of $\displaystyle\lim_{t\to\infty}\frac{1-p_{aa}^t}{1-p_{bb}^t}$ for all $a,b\in I$.

It is well known that in the setting of finite Markov shifts is satisfied the existence of the limit at zero temperature for families of equilibrium states associated to Markov potentials. Below we present a definition of a suitable collection of finite Markov shifts contained into $\Sigma$ which are useful to approximate the unique accumulation point of the family $(\mu_t)_{t\geq1}$ in the weak$^\star$ topology.

\begin{defi}\label{Sigma-c}
	Let $\Sigma$ be a topologically transitive countable Markov shift, $\phi:\Sigma\to\mathbb{R}$ be a Markov potential $\phi(x)=\phi(x_0x_1)$ and let $c\in\mathbb{R}$. We denote by $\Sigma_c$ to the smallest topologically transitive finite Markov shift that contains the symbols $\{i\in\mathbb{N}\colon\sup\phi|_{[i]}\geq c\}$.
\end{defi}
Note that for every $c\in\mathbb{R}$, we have that $\Sigma_c$ is a finite Markov shift  and $\Sigma_c \subset \Sigma_{c\prime}$ when $c \geq c\prime$. In the following, the key argument to prove the existence of the zero temperature limit of equilibrium states on topologically transitive countable Markov shifts, is the construction of an appropriate finite Markov subshift (which remains fixed for all $t\geq1$), whose equilibrium states approximate the ones defined on the countable Markov shift $\Sigma$.

For every $a,b\in I$, we will denote by $\underline{\gamma_a^b}$   the shortest path connecting $a$ to $b$. From the construction, we immediately get that $\underline{\g_a^b}$ contains the symbols $a$ and $b$ only at the ends, so that $\underline{\gamma_a^b} \in \{ \underline{\gamma}: a \rightarrow b \}$. This path always exists because $\Sigma$ is topologically transitive which allows to guarantee that any pair of symbols can be linked by a finite path. We denote $\underline{\gamma_a^b}\underline{\gamma_b^a}$ the concatenation of the paths $\underline{\gamma_a^b}$ with $\underline{\gamma_b^a}$, with this we have that there is always a loop that passes through $a$ and $b$.   We now consider the set 
$$\mathcal{C}:=\bigcup_{a \in I, b \in I} \{\underline{\gamma_a^b}\underline{\gamma_b^a}\in \mathcal{P}(\Sigma)\},$$
which is a non-empty and finite set.
In addition, for each pair $a,b\in I$ and each finite set of symbols $J\subset I$, consider $\underline{\gamma_a^b}(J)$ as the shortest path  connecting $a$ to $b$ and avoiding in the middle any symbol belonging to $J$, the foregoing, in the case that there is at least a path satisfying those conditions. We define the set of all the paths between $a$ and $b$ avoiding some subset of $I$ as
$$ \widetilde{\mathcal{C}}:=\bigcup_{a,b\in I; J\subset I} \{\underline{\gamma_a^b}(J)\in \mathcal{P}(\Sigma)\}.$$
Note that $\widetilde{\mathcal{C}}$ could be an empty set. Finally, we fix
\begin{eqnarray}
	N&:=&\max\{\phantom{-}l(\underline{\gamma}):~\underline{\gamma}\in\mathcal{C}\cup\widetilde{\mathcal{C}}\}<\infty,\label{N}\\
	C&:=&\max\{-\phi(\underline{\gamma}):~\underline{\gamma}\in\mathcal{C}\cup\widetilde{\mathcal{C}}\}<\infty.\label{C}\\
	c&:=&C+\frac{2d}{7}>0,\label{c}
\end{eqnarray}
where $d>0$ was given in the Remark \ref{d}. The above construction leads us to the following proposition.
\begin{proposition}\label{claim} Let $N\in\mathbb{N}$, $C \geq 0$ and $c>0$ as (\ref{N}), (\ref{C}) and (\ref{c}) respectively. Then, for every $a,b\in I$ we have
	\begin{itemize}
		\item[i)]\label{fact1} There exists a loop $\underline{\gamma}\in \mathcal{P}(\Sigma)$ connecting $a$ to $b$ such that $\phi(\underline{\gamma})\geq-C$ and $l(\underline{\gamma})\leq N$;
		\item[ii)]\label{fact2} If there exists a loop $\underline{\gamma}$ connecting $a$ to $a$ and avoiding the set symbols of $J\subset I$, then $\phi(\underline{\gamma})\geq-C$ and $l(\underline{\gamma})\leq N$;
	\end{itemize}
	Also, each symbol of $I$ belongs to the symbols of
	the topologically transitive finite Markov shift $\Sigma_{-c}$, i.e., $\Sigma_I\subset \Sigma_{-c}$.
\end{proposition}
\begin{proof}
	Note that items $i)$ and $ii)$ are a direct consequence of (\ref{N}), (\ref{C}) and (\ref{c}). Furthermore, fixing $i, j\in I$, it follows that
	\begin{eqnarray*}
		-c+\frac{2d}{7}&\leq&\phi(\underline{\gamma_i^j}\underline{\gamma_j^i})\leq\sup\phi|_{[i]},
	\end{eqnarray*}
	so, by Definition \ref{Sigma-c}, we have $\Sigma_I\subset \Sigma_{-c}$.
\end{proof}

The main difference here between the BIP case and the topologically transitive case is the following one: in the BIP case we are able to link $a$ and $b$ using only symbols of the finite set $\{b_1, b_2, ..., b_N\}$ and, thus, the length of the path is at most $N$. On the other hand, in the topologically transitive setting, we are only able to guarantee the existence of a finite path linking $a$ and $b$ but we do not have any control on the length of the paths. Despite this, the finiteness of the set $I$ make it possible to build a compact subshift as in proposition \ref{claim}.

Fix the topologically transitive finite Markov shift $\Sigma^{\prime}:=\Sigma_{-7c}$, note that $\Sigma_I\subset\Sigma_{-c}\subset\Sigma^{\prime}\subset \Sigma$. For the subshift $\Sigma^{\prime}$ we will denote by $\vt_{t}$ to the equilibrium state associated with the potential $t\phi$ restricted to $\Sigma^{\prime}$ (this will be denoted by $t\phi|_{\Sigma^{\prime}}$), and $Q(t\phi) \leq P_G(t \phi)$ is the topological pressure of the potential $t\phi|_{\Sigma^{\prime}}$.

For every $i,j\in\{a,b\}$ and $t\geq 1$, we define:
\begin{equation}\label{q_}
	q_{ij}^{t}:=\sum_{\{\underline{\g} :i\to j\in\mathcal{P}(\Sigma^{\prime})\}}\exp\left((t\phi-Q(t\phi))(\underline{\g})\right).
\end{equation}
where $\{\underline{\g}\colon i\to j\in\mathcal{P}(\Sigma^{\prime})\}$ is the set of paths in $\Sigma^{\prime}$ that connect $i$ to $j$ and such that do not have an intermediate occurrence of neither $a$ nor $b$. Similarly to (\ref{coc-mu}) we see that
\begin{equation}\label{coc-vt}
	\frac{\vt_{t}([b])}{\vt_{t}([a])}=\frac{1-q_{aa}^{t}}{1-q_{bb}^{t}},
\end{equation}
for each $a,b\in I$. Since $\Sigma'$ is a finite Markov shift, we have existence of the limit $\displaystyle\lim_{t\to\infty}\frac{\vt_{t}([b])}{\vt_{t}([a])}$ for any $a,b\in I$, see for details \cite{Br,CGU, Le}. 

Next, we will show the following equality
\begin{equation}\label{equ-import}
	\lim_{t\to\infty}\frac{\mu_t([b])}{\mu_t([a])}=\lim_{t\to\infty}\frac{\vt_{t}([b])}{\vt_{t}([a])},
\end{equation}
for each $a,b\in I$. Since, we have existence of the limit in the right side of the equation above (see for instance \cite{Br,CGU,Le}), by Proposition \ref{acum-max}, it follows that (\ref{equ-import}) guarantees the main theorem of this work (i.e., Theorem \ref{theo-main}).

So, by \eqref{coc-mu} and \eqref{coc-vt}, it is only necessary to prove the following
\begin{equation}\label{equ-import-p}
	\lim_{t\to\infty}\frac{1-p_{aa}^{t}}{1-p_{bb}^{t}}=\lim_{t\to\infty}\frac{1-q_{aa}^{t}}{1-q_{bb}^{t}}.
\end{equation}

To study the asymptotic behavior of $1-p_{aa}^{t}$ and $1-q_{aa}^{t}$ we will use item $i)$ of Proposition \ref{claim} to find their lower bounds. Now we will find a lower bound for $1-p_{aa}^{t}$. Fix $a,b\in I$ and $t\geq 1$,
by Proposition \ref{claim} the concatenation $\underline{\gamma_a^b\gamma_b^a}$ is a path from $a$ to $a$ passing though $b$ with length at most $N$ satisfying $\phi(\underline{\gamma_a^b\gamma_b^a}) \geq -C$, thus
\begin{eqnarray}
	\exp\left((t\phi-P_G(t\phi))(\underline{\gamma_a^b\gamma_b^a})\right)
	\geq\exp\left(-t C-NP_G(t\phi)\right).\label{B}
\end{eqnarray}
Therefore from (\ref{A}) and (\ref{B}), we obtain the lower bound
\begin{equation}\label{2beta}
	1-p_{aa}^{t}=p_{ab}^{t}p_{ba}^{t}\sum_{n\geq 1}(p_{bb}^{t})^n\geq\exp\left(-t C-NP_G(t\phi)\right).
\end{equation}
Obviously the same argument gives us lower bounds for $1-p_{bb}^{t}$, $1-q_{aa}^{t}$ and $1-q_{bb}^{t}$.

In order to continue analyzing the asymptotic behavior of $\frac{1-p_{aa}^{t}}{1-p_{bb}^{t}}$ and $\frac{1-q_{aa}^{t}}{1-q_{bb}^{t}}$, we introduce the terms 
\begin{equation}\label{r}
	r_{ij}^{t}=\sum_{\{\underline{\g}\colon i\to j\in\mathcal{P}(\Sigma^{\prime})\}}\exp(t\phi-P_G(t\phi))(\underline{\g}),
\end{equation}
for each $i,j\in\{a,b\}$. It can be observed that $r_{ij}^{t}\leq p_{ij}^{t}$ and $r_{ij}^{t}\leq q_{ij}^{t}$. Note that $r_{ij}^{t}$ is obtained by taking the sum only on the main paths in $\Sigma^{\prime}$,  while $p_{ij}^{t}$ considers all the main paths of $\Sigma$. Later, in Lemma \ref{lemma1-Kem} we will show that for each  $a\in I$ fixed, the value of  $r_{aa}^{t}$ is close to $p_{aa}^{t}$ and $q_{aa}^{t}$ for $t$ large enough.

The main tool to prove (\ref{equ-import-p}) is the Lemma \ref{lemma1-Kem}, in fact, the Lemmas \ref{af1} and \ref{af23} allow us to prove this lemma. The aforementioned results are the same as those obtained by T. Kempton for countable Markov shifts satisfying the BIP condition and the proofs are similar, those can be found in \cite{Ke}. This is due to Proposition \ref{claim} and because the Gurevich pressure on topologically transitive countable Markov shifts has a similar behavior to the one observed in the case of countable Markov shifts satisfying the BIP condition. This was proved in Proposition \ref{pres-decre}. 

For each main path $\underline{\g}\in\{\underline{\g}:a\to a\}$, we write $n(\underline{\g})$ to denote the number of times that a symbol of $I$ appears in  $\underline{\g}$ without taking into account the symbols that appear at the end. So, when $n(\underline{\g})=n$, these symbols of $I$ can be labeled as $i_0,i_1,i_2,\ldots,i_{n+1}$ with the convention $i_0=i_{n+1}=a$. We denote by $X_{aa}^n$ to all those main paths satisfying $n(\underline{\g})=n$. From the definition of a main path, we necessarily have $i_k \neq \{a,b\}$, $1 \leq k \leq n$. By calling $X_{aa}^n$ to the collection of main paths such that $n(\underline{\g})=n$, we get
\begin{equation}\label{splitpaa}
	p_{aa}^{t}=\sum_{n=0}^{\infty}p_{aa}^{t}(n),
\end{equation}
where
\begin{equation}\label{paan}
	p_{aa}^{t}(n)=\sum_{i_0,\ldots,i_{n+1}\in X_{aa}^n}\prod_{k=0}^{n}\left(\sum_{\{\underline{\g}:i_k\hookrightarrow i_{k+1}\in\mathcal{P}(\Sigma)\}}\exp(t\phi-P_G(t\phi))(\underline{\g})\right).
\end{equation}
Note that for every $n\geq 0$, the terms $p^t_{aa}(n)$ are of the form $\exp(t\phi-P_G(t\phi))(\underline{\g})$, where $\underline{\g}\in\mathcal{P}(\Sigma)$ such that $n(\underline{\g})=n$. The following lemma gives a lower bound for $p_{aa}^{t}(n)$, for every $a\in I$.

\begin{lema}\label{af1} For every $r\in\mathbb{N}_0$ and $r|I|\leq n < (r+1)|I|$, we have that $$p_{aa}^{t}(n)\leq \left(1-\exp(-Ct-NP_G(t\phi))\right)^{r}.$$
\end{lema}

Similarly to \eqref{splitpaa}, \eqref{paan}  we now define $r_{aa}^{t}=\sum_{n=0}^{\infty}r_{aa}^{t}(n)$, where
$$r_{aa}^{t}(n):=\sum_{i_0,\ldots,i_{n+1}\in X_{aa}^n}\prod_{k=0}^{n}\left(\sum_{\{\underline{\g}:i_k\hookrightarrow i_{k+1}\in\mathcal{P}(\Sigma^{\prime})\}}\exp(t\phi-P_G(t\phi))(\underline{\g})\right),$$
By calling
$$\epsilon(n):=\frac{p_{aa}^{t}(n)}{r_{aa}^{t}(n)},$$
it can be checked that
$$0\leq p_{aa}^{t}-r_{aa}^{t}=\sum_{n=0}^{\infty}p_{aa}^{t}(n)\left(1-\frac{1}{\epsilon(n)}\right).$$
\begin{lema}\label{af23} There exists $T>0$ and $K_1$ such that  for all $t>T$
	\begin{itemize}
		\item[i)] For each $0 \leq n<|I|-1$ the following inequality holds,
		$$\epsilon(n)\leq 1+K_1\exp(-5Ct).$$
		\item[ii)] For each $r\geq 1$ and  $r|I|\leq n < (r+1)|I|$ the following statement is satisfied 
		$$\epsilon(n)\leq \left(1+K_1\exp(-5Ct)\right)^r.$$
	\end{itemize}
\end{lema}

\begin{proof}
	From Proposition \ref{pres-decre}, item $ii)$, and Theorem 2 from \cite{FV} the  Gurevich pressure $P_{G}(t\phi)$ decreases to $h_{max}$. So, there exists $T>0$ such that
	$$P_{G}(t\phi)\leq h_{max}+d,\qquad\mbox{for all }t\geq T+1,$$
	where $d$ was given in Remark \ref{d}, also
	$$h_{max} \leq P_G\left(t\phi\right),\qquad\mbox{for all }t\geq1.$$
	Therefore
	\begin{equation*}\label{dif-pre}
		-d\leq P_G\left(t\phi\right)-P_G(T\phi)<0,\qquad\mbox{for all }t\geq T+1,
	\end{equation*}
	thus the difference between the pressure $P_G\left(t\phi\right)$ and $P_G(T\phi)$ can be controlled, for $t\gg 0$. 
	Let $i_k,i_{k+1}\in I$ be arbitrary. Consider a path $\underline{\g}=i_kx_1 \ldots x_{m-1}i_{k+1}:i_k\hookrightarrow i_{k+1}\in \mathcal{P}(\Sigma\setminus\Sigma^{\prime})$,  i.e., $x_n \notin I$ for all $1 \leq n \leq m-1$ and at least one symbol $x_n$, $n=1,2,\ldots,m-1$ does not belongs to the alphabet associated to the finite Markov shift $\Sigma^{\prime}$. From Remark \ref{d} we have $\phi(i_k x_1)\leq 0$,  $\phi(x_n x_{n+1})<-d$ for $n=1,\ldots,m-2$ and $\phi(x_{m-1}i_{k+1})<-d$, because $x_n\notin I$ for $n=1,\ldots,m-1$. Moreover, since $\underline{\g}\in\mathcal{P}(\Sigma\setminus\Sigma^{\prime})$, there exists some $j\in\{1,2,\ldots,m-1\}$ such that $\phi(x_j x_{j+1})<-7c$ and consequently
	$$
	\phi(\underline{\g})=\phi(i_kx_1)+\phi(x_1 x_2)+\ldots+\phi(x_{m-1}i_{k+1}) \leq -d(m-2)-7c.
	$$
	In addition, from \eqref{dif-pre} we have   $P_{G}(T\phi)-P_{G}(t\phi) \leq d \leq (t-T)d$ for $t\geq T+1$, so that
	\begin{eqnarray}
		\left(t\phi-P_{G}(t\phi)\right)(\underline{\g})-\left(T\phi-P_{G}(T\phi)\right)(\underline{\g})&=&\left((t-T)\phi-P_{G}(t\phi)+P_{G}(T\phi)\right)(\underline{\g})\nonumber\\
		&=&m\left(P_{G}(T\phi)-P_{G}(t\phi)\right)+(t-T)\phi(\underline{\g})\nonumber\\
		&\leq&(t-T)md+(t-T)(-d(m-2)-7c)\nonumber\\
		&=&(t-T)(2d-7c)\nonumber\\
		&\leq&-7C(t-T)\label{finineq}.
	\end{eqnarray}
	We now fix the constant 
		\begin{equation}\label{K1}
			K:=\exp(7CT)\max_{i_k,i_{k+1}\in I}\sum_{\underline{\g}:i_k\hookrightarrow i_{k+1}\in\mathcal{P}(\Sigma\setminus\Sigma^{\prime})}\exp\left((T\phi-P_{G}(T\phi))(\underline{\g})\right),
		\end{equation}
		note that $K<\infty$, because $\sum_{\underline{\g}:i_k\hookrightarrow i_{k+1}\in\mathcal{P}(\Sigma\setminus\Sigma^{\prime})}\exp\left((T\phi-P_{G}(T\phi))(\underline{\g})\right)\leq p_{i_ki_{k+1}}^T<\infty$. 
		From \eqref{finineq} we have for $t\geq T+1$
		$$
		\left(t\phi-P_{G}(t\phi)\right)(\underline{\g}) \leq \left(T\phi-P_{G}(T\phi)\right)(\underline{\g}) -7C(t-T),
		$$
		so that
		\begin{eqnarray*}
			\sum_{\underline{\g}:i_k\hookrightarrow i_{k+1}\in\mathcal{P}(\Sigma\setminus\Sigma^{\prime})}\exp\left((t\phi-P_{G}(t\phi))(\underline{\g})\right)
			&\leq&\exp(-7Ct+7CT) \sum_{\underline{\g}:i_k\hookrightarrow i_{k+1}\in\mathcal{P}(\Sigma\setminus\Sigma^{\prime})}\exp\left(((T\phi-P_{G}(T\phi))(\underline{\g})\right)\\
			&\leq& \exp(-7Ct)K.
	\end{eqnarray*}
	
	The proof continues following the same steps as  \cite{Ke} to obtain items $i)$ and $ii)$.
\end{proof}

Due to Lemmas \ref{af1} and \ref{af23} the following lemma is obtained.

\begin{lema}\label{lemma1-Kem}
	There exists $T>0$ and $0<M<\infty$ such that for each pair $a,b\in I$ and for all $t>T$
	\begin{itemize}
		\item[i)] $p_{aa}^{t}\leq r_{aa}^{t}+ M\exp(-3Ct),$
		\item[ii)] $p_{ab}^{t}p_{ba}^{t}\leq r_{ab}^{t}r_{ba}^{t}+ M\exp(-3Ct).$
	\end{itemize}
\end{lema}

We write $a(t)\sim b(t)$, to express that
$$\lim\limits_{t\to\infty}\frac{a(t)}{b(t)}=1.$$
\noindent
As a consequence of the previous lemma, one can obtain that
\begin{eqnarray}\label{lemma2-Kem}
	1-p_{aa}^{t}\sim 1-r_{aa}^{t},\qquad 1-r_{aa}^{t}\sim 1-q_{aa}^{t},
\end{eqnarray}
for every $a\in I$,  see \cite{Ke} for complete details.

Finally, from  (\ref{lemma2-Kem}) we have that
$$\frac{\mu_t([b])}{\mu_t([a])}=\frac{1-p_{aa}^{t}}{1-p_{bb}^{t}}\sim\frac{1-r_{aa}^{t}}{1-r_{bb}^{t}}\sim\frac{1-q_{aa}^{t}}{1-q_{bb}^{t}}=\frac{\vt_{t}([b])}{\vt_{t}([a])}.$$
Therefore,
$$\lim_{t\to\infty}\frac{\mu_t([b])}{\mu_t([a])}=\lim_{t\to\infty}\frac{\vt_{t}([b])}{\vt_{t}([a])},$$
since that $\lim_{t\to\infty}\frac{\vt_{t}([b])}{\vt_{t}([a])}$ exists for all $a,b\in I$ we finally have that $\lim_{t\to\infty}\mu_t$ exists in the weak$^\star$ topology.

\section{Examples on the renewal shift}\label{Examples}

Throughout this section we present some examples where there is selection at zero temperature for the family of equilibrium states $(\mu_t)_{t\geq}$, i.e., where the limit $\lim_{t \to \infty}\mu_t$ exists in the weak$^\star$ topology. In fact, those examples are given in the context of the so called renewal shifts (to be defined below), when the potentials do not necessarily satisfy the conditions stated in Theorem \ref{theo-main}.

The {\it renewal shift} is the countable Markov shift whose transition matrix $\left(A(i,j)\right)_{\mathbb{N}\times \mathbb{N}}$ has entries $A(1,1),A(1,i)$ and $A(i,i-1)$ are equal to $1$ for every $i>1$, and the other entries are equal to $0$. Note that the renewal shift is topologically mixing and does not satisfy the BIP property. 

\begin{figure}[!htb]
	\begin{center}
		\begin{tikzpicture}[scale=1.5]
			\draw   (0,0) -- (7,0);
			\node[circle, draw=black, fill=white, inner sep=1pt,minimum size=5pt] (1) at (0,0) {1};
			\node[circle, draw=black, fill=white, inner sep=1pt,minimum size=5pt] (2) at (1,0) {2};
			\node[circle, draw=black, fill=white, inner sep=1pt,minimum size=5pt] (3) at (2,0) {3};
			\node[circle, draw=black, fill=white, inner sep=1pt,minimum size=5pt] (4) at (3,0) {4};
			\node[circle, draw=black, fill=white, inner sep=1pt,minimum size=5pt] (5) at (4,0) {5};
			\node[circle, draw=black, fill=white, inner sep=1pt,minimum size=5pt] (6) at (5,0) {6};
			\node[circle, draw=black, fill=white, inner sep=1pt,minimum size=5pt] (7) at (6,0) {7};
			\node (8) at (7,0) {};
			\draw[->, >=stealth] (2)  to (1);
			\draw[->, >=stealth] (3)  to (2);
			\draw[->, >=stealth] (4)  to (3);
			\draw[->, >=stealth] (5)  to (4);
			\draw[->, >=stealth] (6)  to (5);
			\draw[->, >=stealth] (7)  to (6);
			\draw[->, >=stealth] (8)  to (7);
			\node [minimum size=10pt,inner sep=0pt,outer sep=0pt] {} edge [in=200,out=100,loop, >=stealth] (0);
			\draw[->, >=stealth] (1)  to [out=90,in=90, looseness=1] (2);
			\draw[->, >=stealth] (1)  to [out=90,in=90, looseness=1] (3);
			\draw[->, >=stealth] (1)  to [out=90,in=90, looseness=1] (4);
			\draw[->, >=stealth] (1)  to [out=90,in=90, looseness=1] (5);
			\draw[->, >=stealth] (1)  to [out=90,in=90, looseness=1] (6);
			\draw[->, >=stealth] (1)  to [out=90,in=90, looseness=1] (7);
		\end{tikzpicture}
	\end{center}
	\caption{Renewal shift.}\label{fig:renewal}
\end{figure}
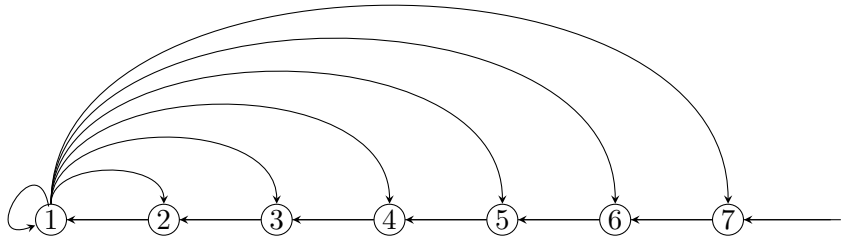

In this subsection we will present 2 examples of the zero temperature limit of equilibrium states on renewal shift. The Example \ref{exem0} is for the potential $\phi(x)=-x_0$, note that this potential satisfies the hypotheses of the Theorem \ref{theo-main}. On the other hand, the Example \ref{exem1} is for the potential $\phi(x)=x_0-x_1$, which is not a summable potential, however it has zero temperature limits for its associated equilibrium states.

In a renewal shift with $\phi:\Sigma\to \mathbb{R}$ a weakly H\"{o}lder continuous function such that $\sup \phi<\infty$, O. Sarig \cite{Sa2} showed that exists $t_c>0$ and a unique equilibrium state $\mu_t\in\mathcal{M}_{\sigma}(\Sigma)$ associated to the potential $t\phi$, for $t\in(0,t_c)$. In that same context, G. Iommi \cite{Io} showed that if $t_c=\infty$ then the set of $\phi$-maximizing $\mathcal{M}_{max}(\phi)\neq\emptyset$, otherwise there is no maximizing measure associated to potential $\phi$.

\begin{example}\label{exem0} Consider the renewal shift $\Sigma$ and a potential $\phi:\Sigma\to \mathbb{R}$ given by  $$\phi(x)={-x_0}.$$
	We will show that $\lim_{t\to\infty}\mu_{t}$ exists and is a maximizing measure, where $\mu_{t}$ is the equilibrium state associated to potential $t\phi$. Also, $\lim_{t\to\infty}\mu_{t}([a])=0$, for all $a\geq 2$.
\end{example}

Note that $\phi$ is a summable potential with  $\overline{\Var}(\phi)=0$. Also, 
$$P_G(t\phi)\leq \log(2)-t,\quad\mbox{for all }t\geq 1.$$
So, by Theorem 1 in \cite{FV} the family $(\mu_{t})_{t\geq 1}$ of equilibrium states associated to potential $t\phi$ have accumulation points. To verify the statement of the previous example we need the following  affirmations to hold:

\begin{afir}\label{afir-exe-3} For $a,n\in\mathbb{N}$, we have
	$$Z_n(t\phi,a)\leq\exp(nP_G(t\phi)).$$
\end{afir}
\begin{proof}
	Let $a,n\in\mathbb{N}$, it can be verified that $Z_{n}\left(t\phi,a\right)=\left(L_{t\phi}^{n}\mathbbm{1}_{[a]}\right)\left(x\right)$ for all $x\in[a]$. Let us integrate that expression with respect to $\nu_t$, which is the eigenmeasure of the dual of the Ruelle operator, i.e.,  $L_{t\phi}^{*}\nu_t=\exp(P_G(t\phi))\nu_t$. So, we obtain
	\begin{align*}
		\begin{aligned}
			Z_n(t\phi,a)&=\frac{1}{\nu_{t}[a]}\int_{[a]}L_{t\phi}^{n}\mathbbm{1}_{[a]}(x)\d\nu_{t}\\
			&\leq\frac{1}{\nu_{t}[a]}\int L_{t\phi}^{n}\mathbbm{1}_{[a]}(x)\d\nu_{t}\\
			&=\frac{\exp(nP_G(t\phi))}{\nu_{t}[a]}\int\mathbbm{1}_{[a]}(x)\d\nu_{t}=\exp(nP_G(t\phi)).
		\end{aligned} 				
	\end{align*}
\end{proof}

\begin{afir}\label{afir-exe-2}For any $a\in\mathbb{N}$, there is a constant $C=\exp\left(-\frac{(a-1)(a+2)}{2}\right)$ such that
	$$\left(L^n_{t\phi}\mathbbm{1}_{[1]}\right)(x)\leq C\cdot Z_{n+a-1}(t\phi,1), $$
	where $x\in[a]$.
\end{afir}
\begin{proof}
	Fix $x\in[a]$, we define the application bijective
	$$\begin{array}{crcl} \theta \ : & \! \ \{y\in[1]:~ \sigma ^{n}\left(y\right)=x\} & \!\longrightarrow & \! \{z\in[1]:~ \sigma ^{n+a-1}(z)=z,z_n=a\} \\ 
		&\! y=1,y_1,\ldots ,y_{n-1},x_0^{\infty} & \! \longmapsto & \! z=\overline{1,y_1,\ldots ,y_{n-1},a,(a-1),\ldots ,2}_{per}. \end{array}$$
	Note that for every $y\in\text{Dom}(\theta)$, $S_n\phi(y)-S_{n+a-1}\phi(\theta(y))=\displaystyle -\frac{(a-1)(a+2)}{2}$. So,
	\begin{align*}
		\begin{aligned}
			\left(L_{t\phi}^n\mathbbm{1}_{[1]}\right)(x)&=\exp\left(-\frac{(a-1)(a+2)}{2}t\right)\sum_{z\in\text{Im}(\theta)}\exp\left(t S_{n+a-1}\phi(z)\right)\\
			&\leq\exp\left(-\frac{(a-1)(a+2)}{2}t\right)\sum_{\sigma^{n+a-1}z=z}\exp\left(t S_{n+a-1}\phi(z)\right)\mathbbm{1}_{[1]}(z)\\
			&=\exp\left(-\frac{(a-1)(a+2)}{2}t\right)\cdot Z_{n+a-1}\left(t\phi,1\right).
		\end{aligned}
	\end{align*}
\end{proof}

\begin{afir}\label{afir-exe-4} For $a\in\mathbb{N}$, we have
	$$\nu_{t}([a])=\exp\left(-\frac{(a+2)(a-1)}{2}t-(a-1)P_G(t\phi)\right)\nu_{t}([1]).$$
\end{afir}
\begin{proof}
	Fixed $t\geq1$, by the Generalized Ruelle's Perron-Frobenius Theorem (see Theorem 4.9 in \cite{Sa3}), there is an eigenmeasure $\nu_{t}$ such that
	\begin{equation}\label{jod0}
		\nu_{t}\left(L_{\phi}f\right)=\lambda\nu_{t}\left( f\right),\qquad\mbox{para } f\in L^1(\nu_{t}).
	\end{equation}
	
	Let $a\geq 2$, consider $f:=\mathbbm{1}_{[a]}$, substituting in $(\ref{jod0})$ we have 
	\begin{align*} 				
		\begin{aligned}
			\exp(P_G(t\phi))\nu_{t}([a])&=\int\sum_{\sigma y=x}\exp({t\phi(y)})\mathbbm{1}_{[a]}(y)\d\nu_{t}(x)\\
			&=\int_{[a-1]}\sum_{\sigma y=x}\exp({t\phi(y)})\mathbbm{1}_{[a]}(y)\d\nu_{t}(x)\\
			&=\exp({-at})\nu_{t}([a-1]).
		\end{aligned} 				
	\end{align*}
	
	So, 
	\begin{equation}\label{rec-exe-imp}
		\nu_{t}([a])=\exp\left(-at-P_G(t\phi)\right)\nu_{t}([a-1]),
	\end{equation}
	using recursively (\ref{rec-exe-imp}) follows the statement.
\end{proof}

Note that the potential $t\phi$ is positive recurrent, for every $t\geq1$. Consider $a\geq 1$, then by the Generalized Ruelle's Perron-Frobenius Theorem,  we have
\begin{align*}
	\begin{aligned}
		\frac{\mu_{t}([a])}{\nu_{t}([a])}&=h(x),\quad \forall x\in[a]\\
		&=\frac{1}{\nu_{t}([1])}\lim_{n\to\infty}\exp(-nP_{G}(t\phi))\left(L^n_{t\phi}\mathbbm{1}_{[1]}\right)(x)\\
		&\leq \frac{\exp\left(-\frac{(a-1)(a+2)}{2}t\right)}{\nu_{t}([1])}\exp((a-1)P_{G}(t\phi)),
	\end{aligned}
\end{align*}
where in the third line the Affirmations were used \ref{afir-exe-3}  and \ref{afir-exe-2}. Later, by Affirmation \ref{afir-exe-4} we have 
\begin{align*}
	\begin{aligned}
		\mu_{t}([a])&\leq\exp\left(-{(a+2)(a-1)}t\right).
	\end{aligned}
\end{align*}
Therefore
\begin{equation}\label{a=0}
	\lim\limits_{t\to\infty}\mu_{t}([a])=0,\qquad\forall a\geq 2.  
\end{equation}
We will show that $\left(\mu_{t}\right)_{t\geq1}$ has only one accumulation point and this is $\delta_{\overline{1}}\in\mathcal{M}_{\sigma}(\Sigma)$, where this measure is the one supported at the point $\overline{1}=111\ldots 1\ldots\in\Sigma$. Consider $\mu_{\infty}\in\mathcal{M}_{\sigma}(\Sigma)$ an arbitrary accumulation point of $\left(\mu_{t}\right)_{t\geq1}$, note that it is enough to show that $\mu_{\infty}=\delta_{\overline{1}}$. From (\ref{a=0}) we have that $\mu_{\infty}([a])=0$, for all $a\geq 2$ and hence
$\mu_{\infty}([1])=1$. Since the measure $\mu_{\infty}$ is $\sigma$-invariant then
$$\mu_{\infty}([1])=\mu_{\infty}([11])+\mu_{\infty}([21]),$$
so, from (\ref{a=0}) and the fact that $[1]=\cup_{i\geq 1}[1i]$ we have that $\mu_{\infty}([1a])=0$, for all $a\geq 2$. Similarly we obtain that
\begin{equation}\label{111a}
	\mu_{\infty}([11\ldots 1a])=0,\qquad\mbox{for all }a\geq 2.  
\end{equation}
So, since $\mu_{\infty}([1])=1$ we have 
\begin{equation}\label{1111}
	\mu_{\infty}([1\ldots1])=1,    
\end{equation}
where $1\ldots1$ is a word of arbitrary size composed only of the symbol one. To show that $\mu_{\infty}=\delta_{\overline{1}}$ it suffices to show that
\begin{equation}\label{igual-delta}
	\mu_{\infty}([\underline{\omega}])=\delta_{\overline{1}}([\underline{\omega}]),\qquad \mbox{for all } \underline{\omega}\in\mathcal{W}.  
\end{equation}
Note that for every $a\geq 2$ and $\underline{\omega}\in\mathcal{W}$ such that $[\underline{\omega}]\cap [a]\neq\emptyset$ we have that \eqref{igual-delta} is satisfied. Also, from (\ref{111a}) and \eqref{1111}, for every $[\underline{\omega}]\cap [1]\neq\emptyset$, we have that \eqref{igual-delta} is satisfied. Therefore $\lim_{t\to\infty}\mu_{t}=\delta_{\overline{1}}$. Also, the Proposition \ref{acum-max} ensures that $\delta_{\overline{1}}\in \mathcal{M}_{max}(\phi)$.

\begin{example}\label{exem1} Consider the renewal shift $\Sigma$ and the potential $\phi:\Sigma\to \mathbb{R}$ given by  $$\phi(x)={x_0-x_1}.$$
	We will show that $\lim_{t\to\infty}\mu_{t}$ exists and it is a $\phi$-maximizing measure, where $\mu_{t}\in\mathcal{M}_{\sigma}(\Sigma)$ is the equilibrium state associated to potential $t\phi$, for every $t\geq 1$. Also, $\lim_{t\to\infty}\mu_{t}([a])>0$, for every $a\in\mathbb{N}$.
\end{example}

Note that $\phi$ be a weakly H\"{o}lder continuous potential with $\Var_1(\phi)=+\infty$ and $\sup\phi<\infty$. Also $t_c=\infty$, next, by Theorem 5 in \cite{Sa2} we know that it exists the equilibrium state $\mu_{t}\in\mathcal{M}_{\sigma}(\Sigma)$ associated with the potential $t\phi$, for every $t\geq1$, and hence $\mathcal{M}_{max}(\phi)\neq\emptyset$, see Theorem 1.1 in \cite{Io}. Fix $t\geq1$, and notice that the Gurevich pressure of $P_G(t\phi)$ is constant, $P_G(t\phi)=\log2$ (see \cite{BBE}). Also, $\alpha(\phi)=0$,
and $|\mathcal{M}_{max}(\phi)|=\infty$, because of the fact that for every $x\in Per(\Sigma)$ such that $\sigma^nx=x$ we have $S_n\phi(x)=0$.

On the other hand, let $\underline{\omega}\in\mathcal{W}_m$, such that $\underline{\omega}\subset[a]$ then
\begin{eqnarray}
	\mu_{t}([\underline{\omega}])&=&h_{t}(x)\nu_{t}([\underline{\omega}]),\qquad x\in[a]\nonumber\\
	&=&\lim_{n\to\infty}\frac{1}{2^n}L_{t\phi}^n\mathbbm{1}_{[\underline{\omega}]}(x)\nonumber\\
	&=&\lim_{n\to\infty}\frac{1}{2^n}\left\{y\in[\underline{\omega}]:~\sigma^ny=x\right\}.\label{omega=a}
\end{eqnarray}
So, for any cylinder $\underline{\omega}$ we have that $\mu_{t}([\underline{\omega}])$ does not depend on $t$. Since $\underline{\omega}$ was arbitrary, we have that $(\mu_{t})_{t\geq1}$ is a singleton which we will denote by $\mu_{\infty}\in\mathcal{M}_{\sigma}(\Sigma)$. By Proposition \ref{pres-decre}, item $iii)$, $h(\mu_{\infty})=\log 2$. So, by the variational principle we have
$$P_G(t\phi)=h(\mu_{\infty})+\mu_{\infty}(t\phi),$$
thus $\mu_{\infty}(\phi)=0$. Therefore, $\mu_{\infty}$ is a $\phi$-maximizing measure. If consider $\underline{\omega}=a$ in (\ref{omega=a}) we have $\mu_{\infty}([a])=\frac{1}{2^a}$.

Note that from Example \ref{exem1} we have the existence of the zero temperature limit of equilibrium states in more general conditions than Theorem \ref{theo-main}.

\section*{Acknowledgements}
\addcontentsline{toc}{section}{Acknowledgements}

VV would to thank to the Foundation for Science and Technology (FCT) Project UIDP/00144/2020 and the Francisco Jos\'e de Caldas Fund (FFJC) Process 80740-628-2020 by the financial support during part of the development of this paper. CM thanks CONACYT Mexico for financial support through Ciencia de Frontera (Ciencia B\'{a}sica) Project No. A1-S-15528. JL and EB would to thank to the fellow program Fondo Postdoctorado Universidad Cat\'olica del Norte No 0001, 2020.


\begin{thebibliography}{{Kem}11}

\bibitem[BBE21]{BBE}
E.~R {Beltr{\'a}n}, R. {Bissacot}, and E.~O {Endo}.
\newblock Infinite dlr measures and volume-type phase transitions on countable
  markov shifts.
\newblock {\em Nonlinearity}, 34(7):4819, 2021.

\bibitem[BF14]{BF}
R.~{Bissacot} and R.~{Freire}.
\newblock {On the existence of maximizing measures for irreducible countable
  Markov shifts: a dynamical proof}.
\newblock {\em {Ergodic Theory Dyn. Syst.}}, 34(4):1103--1115, 2014.

\bibitem[BLL13]{BLL13}
A.~{Baraviera}, R.~{Leplaideur}, and A.~O. {Lopes}.
\newblock {\em Ergodic optimization, zero temperature limits and the max-plus
  algebra. Paper from the 29th Brazilian mathematics colloquium -- 29\(^{\text
  o}\) Col\'oquio Brasileiro de Matem\'atica, Rio de Janeiro, Brazil, July 22
  -- August 2, 2013}.
\newblock Rio de Janeiro: Instituto Nacional de Matem\'atica Pura e Aplicada
  (IMPA), 2013.

\bibitem[BMP16]{BM}
R.~{Bissacot}, J.~{Mengue}, and E.~{P{\'e}rez}.
\newblock A large deviation principle for gibbs states on markov shifts at zero
  temperature.
\newblock 2016.
\newblock arXiv:1612.05831.

\bibitem[{Bow}75]{Bow75}
R.~{Bowen}.
\newblock {\em {Equilibrium states and the ergodic theory of Anosov
  diffeomorphisms}}, volume 470.
\newblock Springer, Cham, 1975.

\bibitem[Br{\'e}03]{Br}
J.~{Br{\'e}mont}.
\newblock Gibbs measures at temperature zero.
\newblock {\em Nonlinearity}, 16(2):419--426, 2003.

\bibitem[BS03]{BS}
J.~{Buzzi} and O.~{Sarig}.
\newblock {Uniqueness of equilibrium measures for countable Markov shifts and
  multidimensional piecewise expanding maps}.
\newblock {\em {Ergodic Theory Dyn. Syst.}}, 23(5):1383--1400, 2003.

\bibitem[CF90]{Coe90}
Z.~{Coelho-Filho}.
\newblock {\em Entropy and ergodicity of skew-products over subshifts of finite
  type and central limit asymptotics}.
\newblock PhD thesis, University of Warwick, 1990.

\bibitem[CGU11]{CGU}
J.-R. {Chazottes}, J.-M. {Gambaudo}, and E.~{Ugalde}.
\newblock {Zero-temperature limit of one-dimensional Gibbs states via
  renormalization: the case of locally constant potentials}.
\newblock {\em {Ergodic Theory Dyn. Syst.}}, 31(4):1109--1161, 2011.

\bibitem[CH10]{CH}
J.-R. {Chazottes} and M.~{Hochman}.
\newblock {On the zero-temperature limit of Gibbs states}.
\newblock {\em {Commun. Math. Phys.}}, 297(1):265--281, 2010.

\bibitem[{Cyr}11]{Cyr}
V.~{Cyr}.
\newblock {Countable Markov shifts with transient potentials}.
\newblock {\em {Proc. London Math. Soc.}}, 103(3):923--949, 2011.

\bibitem[FV18]{FV}
R.~{Freire} and V.~{Vargas}.
\newblock {Equilibrium states and zero temperature limit on topologically
  transitive countable Markov shifts}.
\newblock {\em {Trans. Am. Math. Soc.}}, 370(12):8451--8465, 2018.

\bibitem[{Gur}84]{Gu}
B.~M. {Gurevich}.
\newblock {A variational characterization of one-dimensional countable state
  Gibbs random fields}.
\newblock {\em {Z. Wahrscheinlichkeitstheor. Verw. Geb.}}, 68:205--242, 1984.

\bibitem[{Iom}07]{Io}
G.~{Iommi}.
\newblock {Ergodic optimization for renewal type shifts}.
\newblock {\em {Monatsh. Math.}}, 150(2):91--95, 2007.

\bibitem[JMU05]{JMU1}
O.~{Jenkinson}, R.~D. {Mauldin}, and M.~{Urba\'nski}.
\newblock {Zero temperature limits of Gibbs-equilibrium states for countable
  alphabet subshifts of finite type}.
\newblock {\em {J. Stat. Phys.}}, 119(3-4):765--776, 2005.

\bibitem[JMU06]{JMU2}
O.~{Jenkinson}, R.~D. {Mauldin}, and M.~{Urba\'nski}.
\newblock {Ergodic optimization for countable alphabet subshifts of finite
  type}.
\newblock {\em {Ergodic Theory Dyn. Syst.}}, 26(6):1791--1803, 2006.

\bibitem[{Kem}11]{Ke}
T.~{Kempton}.
\newblock {Zero temperature limits of Gibbs equilibrium states for countable
  Markov shifts}.
\newblock {\em {J. Stat. Phys.}}, 143(4):795--806, 2011.

\bibitem[{Lep}05]{Le}
R.~{Leplaideur}.
\newblock {A dynamical proof for the convergence of Gibbs measures at
  temperature zero}.
\newblock {\em {Nonlinearity}}, 18(6):2847--2880, 2005.

\bibitem[LV21]{LV21}
A.~O. {Lopes} and V.~{Vargas}.
\newblock Entropy, pressure, ground states and calibrated sub-actions for
  linear dynamics.
\newblock 2021.
\newblock arXiv:2105.00078.

\bibitem[{Mor}07]{Mo}
I.~D. {Morris}.
\newblock {Entropy for zero-temperature limits of Gibbs-equilibrium states for
  countable-alphabet subshifts of finite type}.
\newblock {\em {J. Stat. Phys.}}, 126(2):315--324, 2007.

\bibitem[MU01]{MU}
R.~D. {Mauldin} and M.~{Urba\'nski}.
\newblock {Gibbs states on the symbolic space over an infinite alphabet}.
\newblock {\em {Isr. J. Math.}}, 125:93--130, 2001.

\bibitem[{Rue}68]{Ru}
D.~{Ruelle}.
\newblock {Statistical mechanics of a one-dimensional lattice gas}.
\newblock {\em {Commun. Math. Phys.}}, 9:267--278, 1968.

\bibitem[{Sar}99]{Sa1}
O.~M. {Sarig}.
\newblock {Thermodynamic formalism for countable Markov shifts}.
\newblock {\em {Ergodic Theory Dyn. Syst.}}, 19(6):1565--1593, 1999.

\bibitem[{Sar}01]{Sa2}
O.~M. {Sarig}.
\newblock {Phase transitions for countable Markov shifts}.
\newblock {\em {Commun. Math. Phys.}}, 217(3):555--577, 2001.

\bibitem[Sar09]{Sa3}
O.~M. {Sarig}.
\newblock Lecture notes on thermodynamic formalism for topological markov
  shifts.
\newblock {\em Penn State}, 2009.

\bibitem[{Shw}19]{Sh}
O.~{Shwartz}.
\newblock {Thermodynamic formalism for transient potential functions}.
\newblock {\em {Commun. Math. Phys.}}, 366(2):737--779, 2019.

\bibitem[SV22]{SV}
R.~R. {Souza} and V.~{Vargas}.
\newblock Existence of gibbs states and maximizing measures on a general
  one-dimensional lattice system with markovian structure.
\newblock {\em Qual. Theory Dyn. Syst.}, 21(1):5, 2022.

\bibitem[vR07]{vER}
A.~C.~D. {van Enter} and W.~M. {Ruszel}.
\newblock {Chaotic temperature dependence at zero temperature}.
\newblock {\em {J. Stat. Phys.}}, 127(3):567--573, 2007.

\end{thebibliography}
\end{document}